\documentclass[11pt]{article}

\usepackage[utf8]{inputenc}
\usepackage[T1]{fontenc}
\usepackage{mathpazo}
\usepackage{comment}
\RequirePackage[english]{babel}

\RequirePackage[a4paper,margin=2.5cm]{geometry}
\RequirePackage{parskip}

\newcommand{\affiliation}{\footnote}
\makeatletter
\def\@fnsymbol#1{\ensuremath{\ifcase#1\or *\or \dagger\or \ddagger\or \mathsection\or \|\or **\or \dagger\dagger \or \ddagger\ddagger \else\@ctrerr\fi}}
\makeatother

\usepackage{xcolor}
\definecolor{cblue}{RGB}{0,70,140}
\definecolor{cgreen}{RGB}{100,140,0}
\definecolor{cred}{RGB}{190,10,50}

\usepackage[shortlabels]{enumitem}
\setlist{itemsep=0ex,topsep=0ex,parsep=0.4ex}

\usepackage{amsmath,amsfonts,amssymb,amsthm,mathtools,bbm,centernot}

\usepackage[hyphens]{url} 
\usepackage[hidelinks,colorlinks,pagebackref]{hyperref} 
\hypersetup{citecolor=cgreen,linkcolor=cblue,urlcolor=cblue}
\usepackage[capitalise,nameinlink,noabbrev,compress]{cleveref} 

\renewcommand*{\backref}[1]{}
\renewcommand*{\backrefalt}[4]{
	\ifcase #1 Not cited.%
	\or $\uparrow$#2%
	\else $\uparrow$#2%
	\fi%
}

\let\oldbibliography\bibliography
\renewcommand{\bibliography}[1]{
  {

    \hypersetup{linkcolor=cred}
    \bibliographystyle{bibstyle}
    \oldbibliography{#1}
  }
}

\theoremstyle{plain}
\newtheorem{theorem}{Theorem}[section]
\newtheorem{lemma}[theorem]{Lemma}

\newtheorem{proposition}[theorem]{Proposition}

\newtheorem{corollary}[theorem]{Corollary}

\newtheorem{remark}[theorem]{Remark}
\theoremstyle{definition}

\makeatletter
\renewenvironment{proof}[1][\proofname]
{\par\pushQED{\qed}
	\normalfont\topsep6\p@\@plus6\p@\relax\trivlist
	\item[\hskip\labelsep\bfseries#1\@addpunct{.}]
	\ignorespaces}
{\popQED\endtrivlist\@endpefalse}
\makeatother

\DeclareMathOperator{\diam}{diam}

\title{Asymptotically attaining the Moore bound}
\author{
  W. Cames van Batenburg%
  \affiliation{%
    D\'epartement d'Informatique,
    Universit\'e libre de Bruxelles, Belgium
    (\textsf{\href{mailto:w.p.s.camesvanbatenburg@gmail.com}%
    {w.p.s.camesvanbatenburg@gmail.com}}).
    Supported by the Belgian National Fund for Scientific Research (FNRS).
  }
  \and
  S. Korsky%
  \affiliation{%
    Independent researcher,
    (\textsf{\href{mailto:samkorsky11@gmail.com}{samkorsky11@gmail.com}}).
  }
}
\date{\today}

\begin{document}
\maketitle
\begin{abstract}
For positive integers $d$ and $k$, let $n_k(d)$ be the maximum order of a graph of maximum degree at most $d$ and diameter at most $k$. We prove that
$$ \lim_{d\to\infty}\frac{n_k(d)}{d^k}=1$$
for every fixed $k$, thereby resolving the asymptotic degree--diameter problem for fixed diameter and proving a conjecture of Bollob\'as. The lower bound comes from regular graphs $H_{k,q}$, indexed by prime powers $q$, whose vertices are partial flags in
$\mathbb{F}_q^{\,2k+1}$. These graphs have diameter $k$ and order $|V(H_{k,q})| =(1+o(1))\Delta(H_{k,q})^k$. We also construct, for every fixed $\ell \ge 2$, graphs of maximum degree at most $d$ and line-graph diameter at most $\ell$ with $(1+o(1))d^{\ell}$ edges.
\end{abstract}

\section{Introduction}

Let $k$ and $d$ be positive integers. At most how many vertices can a graph $G$ have if its maximum degree $\Delta(G)$ is at most $d$ and its diameter $\diam G$ is at most $k$?  In this work we obtain an asymptotically tight answer to that question. Define
\begin{equation}\label{eq:nkd}
  n_k(d):=\max\{|V(G)|:\Delta(G)\le d\text{ and }\diam G\le k\}.
\end{equation}
 A breadth-first search from any vertex gives the \emph{Moore bound}
\begin{equation}\label{eq:moore}
  n_k(d)\le 1+d\sum_{j=0}^{k-1}(d-1)^j
  =d^k+O_k(d^{k-1}).
\end{equation}
A graph attaining equality in the Moore bound is called a \emph{Moore graph}.
Complete graphs attain the bound when $k=1$, and the odd cycle $C_{2k+1}$ attains it when $d=2$.  Apart from these families and the trivial degree-one cases, the Petersen graph and the Hoffman--Singleton graph attain the bound for $(d,k)=(3,2)$ and $(7,2)$, respectively.  The only remaining potential nontrivial Moore graph would have degree $57$ and diameter $2$; no Moore graph exists for any other parameters~\cite{HS,BI,Dam}.  We refer to the survey of Miller and \v{S}ir\'a\v{n}~\cite{MS} for further background on the degree--diameter problem.

The rarity of exact equality leads naturally to the asymptotic problem in which $k$ is fixed and $d$ tends to infinity.  Bollob\'as conjectured that, for every fixed $k\ge2$ and every $\varepsilon>0$, there are infinitely many $d$ for which
$$ n_k(d)\ge(1-\varepsilon)d^k;$$
equivalently, $\limsup_{d\to\infty}\frac{n_k(d)}{d^k}=1$; see the discussion in~\cite[p.~2]{CCJK}. This longstanding conjecture was previously known only for $k\in\{2,3,5\}$, through constructions arising from finite generalized polygons~\cite{Del1,Del2,BSS}.  For all sufficiently large $k$, Canale and G\'omez proved that infinitely many degrees $\Delta$ admit a graph of diameter $k$ and order greater than $(0.629\,\Delta)^k$~\cite{CG}; see also~\cite[Proposition~5]{CCJK}.  This was the strongest coefficient known uniformly over all sufficiently large diameters. 
Another variant of Bollob\'as' conjecture~\cite{Bol, Bol_later}
is that $\liminf_{d\to\infty}\frac{n_k(d)}{d^k}=1$. We prove that stronger statement.

\begin{theorem}\label{thm:vertex}
For every fixed positive integer $k$,
\begin{equation}\label{eq:vertex-limit}
  \lim_{d\to\infty}\frac{n_k(d)}{d^k}=1.
\end{equation}
\end{theorem}


Our construction uses complete flags in $\mathbb F_q^{2k+1}$.  The vertices are the even-rank subflags, and two vertices are adjacent when some odd-rank subflag extends both to complete flags. 
Any two complete flags are induced by two orderings of a common basis. Odd-even transposition sorting (alternating rounds of disjoint adjacent swaps) transforms one ordering into the other, and since only every second round changes the even ranks, the corresponding vertices are joined by a path of length at most $k$. Each vertex has exactly $(q + 1)^k$ extensions to a complete flag, yielding a natural upper bound on the degree of the graph. The Moore bound shows that this upper bound is asymptotically tight, and the prime number theorem supplies suitable field sizes for all sufficiently large degree bounds.

We also consider the edge version of the problem, posed by Erd\H{o}s and Ne\v{s}et\v{r}il~\cite{Erdos88} and studied recently by Cambie, Cames van Batenburg, de Joannis de Verclos, and Kang~\cite{CCJK}.  For positive integers $\ell$ and $d$, write
\begin{equation}\label{eq:hell}
  h_\ell(d)-1:=\max\{|E(G)|:\Delta(G)\le d
  \text{ and }\diam L(G)\le\ell\},
\end{equation}
where $L(G)$ is the line graph of $G$.  Cambie et al.\ conjectured that, for every fixed $\ell$ and every $\varepsilon>0$, $h_\ell(d)\ge(1-\varepsilon)d^\ell$ for infinitely many $d$
\cite[Conjecture~3]{CCJK}.  We prove the stronger liminf statement.

\begin{corollary}\label{cor:edge}
For every fixed integer $\ell\ge2$,
\begin{equation}\label{eq:edge-limit}
  \liminf_{d\to\infty}\frac{h_\ell(d)}{d^\ell}\ge1.
\end{equation}
\end{corollary}

The graphs used to prove Corollary~\ref{cor:edge} are bipartite.  Together with the matching upper bound for odd-cycle-free graphs in
\cite[Theorem~7]{CCJK}, this also shows that the limit is $1$ when restricted to bipartite graphs.
In contrast, for general graphs the bound is not tight. For $\ell=2$, it is known~\cite{CGTT90} that $\frac{h_2(d)}{d^2}$ converges to $\frac{5}{4}$, which is attained by blown-up $5$-cycles and is related to the infamous Erd\H{o}s--Ne\v{s}et\v{r}il conjecture on coloring the square of a line graph; see~\cite{Cranston26, CvBKP20, DHJKV26} for more background. For $\ell=3$, Kumar, Mohar and Pragada~\cite{KMP26} recently showed that $\liminf_{d\to\infty}\frac{h_3(d)}{d^3}\ge \frac{253}{225}$ is strictly larger than $1$. For $\ell \in\{4,6\}$, the lower bound $1$ was known to follow from geometric constructions similar to those used for the vertex version. For all sufficiently large values of $\ell$, our Corollary~\ref{cor:edge} improves the lower bound from $ (0.629)^{\ell}$ to $1$. For all $\ell\ge 3$, the current best upper bound~\cite{CCJK} is $3/2$.

\section{Alternating routes between complete flags}

Let $V$ be a $t$-dimensional vector space.  A complete flag in $V$ is an increasing chain of subspaces
$$F=(0=F_0<F_1<\cdots<F_{t-1}<F_t=V),\qquad \dim F_i=i.$$
Write
$$F_{\mathrm{odd}}
  :=(F_i)_{\substack{1\le i<t\\ i\text{ odd}}}, \qquad F_{\mathrm{even}}
  :=(F_i)_{\substack{1\le i<t\\ i\text{ even}}}.$$
\begin{lemma}\label{lem:routing}
Let $F$ and $F'$ be complete flags in a $t$-dimensional vector space.  There is a sequence of complete flags
$$F=G^0,G^1,\ldots,G^t=F'$$
such that, at odd-numbered steps, only odd-rank members may change, and at even-numbered steps, only even-rank members may change.
\end{lemma}

\begin{proof}

By the standard fact that any two complete flags are induced by two orderings of a common basis (see the proof of axiom {\rm (B1)} in \cite[Section~4.3, pp.~183--184]{AB}), there are a basis $v_1,\ldots,v_t$ and a permutation $\pi\in S_t$ such that
$$  F_i=\operatorname{span}(v_1,\ldots,v_i),\qquad
  F'_i=\operatorname{span}(v_{\pi(1)},\ldots,v_{\pi(i)})
  \quad(0\le i\le t).$$
Label $v_i$ by its target position $\pi^{-1}(i)$ and apply odd--even transposition sorting to these labels. This transforms the order $v_1,\ldots,v_t$ into $v_{\pi(1)},\ldots,v_{\pi(t)}$. Allowing stationary steps, it uses $t$ alternating rounds~\cite[p.~518]{ACG}. In an odd round one swaps a subset of the disjoint adjacent pairs
$$(1,2),(3,4),\ldots,$$
and in an even round one swaps a subset of
$$(2,3),(4,5),\ldots.$$
Swapping the vectors in positions $i$ and $i+1$ changes only the
$i$-dimensional prefix span.  An odd round therefore changes only odd-rank members of the associated flag, and an even round changes only even-rank members.  The associated flags give the required route.
\end{proof}

For a $t$-dimensional vector space $V$, let $\mathcal F(V)$ be its set of
complete flags, and set
$$\mathcal P_{\mathrm{odd}}(V)
  :=\{F_{\mathrm{odd}}:F\in\mathcal F(V)\},
  \qquad
  \mathcal P_{\mathrm{even}}(V)
  :=\{F_{\mathrm{even}}:F\in\mathcal F(V)\}.$$
An odd-rank partial flag $F_{\mathrm{odd}}$ and an even-rank partial flag $F_{\mathrm{even}}$ are called \emph{compatible} if their members, interleaved by rank, form a complete flag.  A compatible pair determines that complete flag uniquely.

For a prime power $q$, write
 $$[j]_q:=1+q+\cdots+q^{j-1},\qquad
  [t]_q!:=\prod_{j=1}^t[j]_q.$$
Thus $[t]_q!$ is the number of complete flags in $\mathbb F_q^t$.

\section{The halved flag construction}

Fix $k\ge1$, let $t=2k+1$, and take $V=\mathbb F_q^t$.  Define $H_{k,q}$ to be the graph with vertex set $\mathcal P_{\mathrm{even}}(V)$ in which two distinct vertices are adjacent if they have a common compatible member of $\mathcal P_{\mathrm{odd}}(V)$. Equivalently, $H_{k,q}$ is obtained by taking the square of the bipartite compatibility graph between $\mathcal P_{\mathrm{even}}(V)$ and $\mathcal P_{\mathrm{odd}}(V)$, and then restricting to the vertices that represent $\mathcal P_{\mathrm{even}}(V)$. Via the last step we only retain half the vertices, so in this sense we have ``halved the flags.''

\begin{proposition}\label{prop:H}
For every prime power $q$ and every $k\ge1$, the graph $H_{k,q}$ is regular. Writing $\Delta_q$ for its degree, we have
\begin{equation}\label{eq:H-order}
  N_q:=|V(H_{k,q})|=\frac{[2k+1]_q!}{(q+1)^k},
\end{equation}
\begin{equation}\label{eq:H-degree}
  \Delta_q\le K_q:=(q+1)^k\left((q+1)^k-1\right),
\end{equation}
and
\begin{equation}\label{eq:H-diameter}
  \diam H_{k,q}=k.
\end{equation}
For fixed $k$, as $q\to\infty$ through prime powers,
\begin{equation}\label{eq:H-asymptotic}
  \Delta_q=(1+o(1))K_q,
  \qquad
  N_q=(1+o(1))\Delta_q^k.
\end{equation}
\end{proposition}

\begin{proof}
\smallskip
\noindent\emph{Order and degree.}
There are $k$ odd and $k$ even ranks among $1,\ldots,t-1$.  Fix
$P_{\mathrm{even}}\in\mathcal P_{\mathrm{even}}(V)$.  For each odd
$i\in\{1,\ldots,t-1\}$, a completion to a full flag must choose an
intermediate space
$$(P_{\mathrm{even}})_{i-1}<F_i<(P_{\mathrm{even}})_{i+1},$$
where $F_0=0$ and $F_t=V$ are understood.  The corresponding quotient has dimension two, and hence there are $q+1$ choices for $F_i$.  These choices are independent, so every even-rank partial flag has exactly $(q+1)^k$ completions.  The same argument applies to odd-rank partial flags.  Counting complete flags by their even-rank parts proves~\eqref{eq:H-order}.

Any two even-rank partial flags are carried to one another by an invertible linear map.  Since invertible linear maps preserve compatibility, they induce automorphisms of $H_{k,q}$.  Thus $H_{k,q}$ is vertex-transitive and hence regular.  From a fixed vertex there are exactly $(q+1)^k$ compatible odd-rank partial flags, and each of these is compatible with exactly $(q+1)^k-1$ other even-rank partial flags.  Different odd-rank partial flags may yield the same neighbor, so \eqref{eq:H-degree} follows.

\smallskip
\noindent\emph{Diameter.}
Take $P_{\mathrm{even}},P'_{\mathrm{even}}
\in\mathcal P_{\mathrm{even}}(V)$ and complete them to full flags $F$ and $F'$.  Apply Lemma~\ref{lem:routing} to obtain an odd-first route
$$F=G^0,G^1,\ldots,G^t=F'.$$
For $0\le j\le k$, set
$$P_j:=(G^{2j})_{\mathrm{even}}
      =(G^{2j+1})_{\mathrm{even}},$$
and, for $1\le j\le k$, set
$$Q_j:=(G^{2j-1})_{\mathrm{odd}}
      =(G^{2j})_{\mathrm{odd}}.$$
The equality for $P_j$ holds because odd steps leave all even-rank
members unchanged, while the equality for $Q_j$ holds because even
steps leave all odd-rank members unchanged. Thus $Q_j$ is compatible
with $P_{j-1}$ via $G^{2j-1}$ and with $P_j$ via $G^{2j}$. Hence
$$  P_{j-1} \text{ and } P_j \text{ are either equal or adjacent in } H_{k,q}.$$
Finally, since $G^0=F$ and $G^t=F'$, we have
$P_0=P_{\mathrm{even}}$ and $P_k=P'_{\mathrm{even}}$.
Deleting repetitions gives a path of length at most $k$.

For the reverse inequality, fix a basis $e_1,\ldots,e_t$.  For an odd- or even-rank partial flag $R=(R_i)$, put
$$\lambda(R):=\max\left(\{i:e_1\notin R_i\}\cup\{0\}\right).$$
Thus $\lambda(R)$ is the last recorded rank before $e_1$ appears.  If
$R\in\mathcal P_{\mathrm{even}}(V)$ and $S\in\mathcal P_{\mathrm{odd}}(V)$ are compatible, let $r$ be the first rank at which $e_1$ occurs in their common complete flag.  Then $\lambda(R)$ and $\lambda(S)$ are the largest even and odd ranks below $r$, respectively, so
$$|\lambda(R)-\lambda(S)|\le1.$$
It follows that adjacent vertices of $H_{k,q}$ have $\lambda$-values
differing by at most two.  For the standard and opposite complete flags
$$A_i=\operatorname{span}(e_1,\ldots,e_i),
  \qquad
  C_i=\operatorname{span}(e_t,e_{t-1},\ldots,e_{t+1-i}),$$
one has
$$\lambda(A_{\mathrm{even}})=0,
  \qquad
  \lambda(C_{\mathrm{even}})=2k.$$
A path of length $s$ between these vertices would therefore satisfy $2k\le2s$, so $s\ge k$.  This proves~\eqref{eq:H-diameter}.

\smallskip
\noindent\emph{Asymptotics.}
For fixed $k$, we have
$[2k+1]_q!
=q^{k(2k+1)}\left(1+O_k(q^{-1})\right).$
Equations~\eqref{eq:H-order} and~\eqref{eq:H-degree} give $N_q=q^{2k^2}\left(1+O_k(q^{-1})\right)$ and $K_q=q^{2k}\left(1+O_k(q^{-1})\right)$, 
and hence
\begin{equation}\label{eq:N-K-asymptotic}
  N_q=\left(1+O_k(q^{-1})\right)K_q^k.
\end{equation}
Since $N_q\to\infty$, the Moore bound forces $\Delta_q\to\infty$, and hence $N_q\le(1+o(1))\Delta_q^k$.
Together with~\eqref{eq:N-K-asymptotic} and $\Delta_q\le K_q$, this gives $\Delta_q\sim K_q$, and therefore $N_q\sim\Delta_q^k$.  This proves \eqref{eq:H-asymptotic}.
\end{proof}

\begin{proof}[Proof of Theorem~\ref{thm:vertex}]
The case $k=1$ follows from the complete graph $K_{d+1}$.  Fix $k\ge2$ and let $d\to\infty$.  Let $q$ be the largest prime not exceeding $d^{1/(2k)}-1$.  The prime number theorem gives $q=(1+o(1))d^{1/(2k)}$.
For this choice,
$$\Delta(H_{k,q})\le K_q<(q+1)^{2k}\le d,
  \qquad
  K_q=(1+o(1))d.$$
Proposition~\ref{prop:H} therefore yields
$$n_k(d)\ge N_q=(1+o(1))K_q^k=(1+o(1))d^k.$$
Together with the Moore bound~\eqref{eq:moore}, this proves
\eqref{eq:vertex-limit}.
\end{proof}

\begin{remark}
For fixed $k\ge2$, the explicit degree bound in Proposition~\ref{prop:H} gives
$$\frac{|V(H_{k,q})|^{1/k}}{\Delta(H_{k,q})}
  \ge \frac{N_q^{1/k}}{K_q}
  =1-\frac{2k-1}{q}+O_k(q^{-2}).$$
More involved variants can improve the coefficient of $q^{-1}$, but we have chosen the halved construction because it gives the shortest and most transparent proof. It would be interesting to determine, for fixed $k$, the optimal rate at which $n_k(d)/d^k$ tends to one. This is particularly relevant in light of the following two questions asked by Bermond and Bollob\'as~\cite{BB81}, which remain open. Let
$M_k(d) = 1 + d\sum_{j = 0}^{k - 1}(d - 1)^j$
denote the Moore bound. Does $\frac{n_k(d)}{M_k(d)}$ converge to $1$ along every sequence of pairs $(k,d)$ for which both $k$ and $d$ tend to infinity, and is $M_k(d) - n_k(d)$ unbounded?
\end{remark}

\section{The edge version}

Given a graph $H$, define a bipartite graph $B(H)$ with parts $V(H)_0$ and $V(H)_1$ by joining $x_0$ to $y_1$ whenever $x=y$ or $xy\in E(H)$.

\begin{lemma}\label{lem:edge-reduction}
For every graph $H$,
$$ \diam L(B(H))\le\diam H+1.$$
\end{lemma}

\begin{proof}
Let $\diam H\le k$, and write an edge $x_0y_1$ of $B(H)$ as $[x,y]$.
Take two edges $e=[a,b]$ and $f=[c,d]$.  If $k$ is odd, choose a sequence
$$b=z_0,z_1,\ldots,z_r=c$$
of odd length $r\le k$ in which consecutive terms are equal or adjacent in $H$.  If $k$ is even, choose such a sequence from $b$ to $d$ of even length $r\le k$.  Such a sequence exists by taking a shortest path and, if necessary, repeating one vertex once.

Starting from $[a,z_0]$, alternately replace the first and second coordinates:
$$[a,z_0],[z_1,z_0],[z_1,z_2],[z_3,z_2],\ldots.$$
Every displayed pair is an edge of $B(H)$, and consecutive pairs share an endpoint.  After $r$ replacements, the last edge shares its $0$-side
endpoint with $f$ when $r$ is odd, and its $1$-side endpoint with $f$ when $r$ is even.  Thus $e$ and $f$ are at distance at most $r+1\le k+1$ in $L(B(H))$.
\end{proof}

\begin{proof}[Proof of Corollary~\ref{cor:edge}]
Fix $\ell\ge2$, put $k=\ell-1$, and let $d\to\infty$.  Let $q$ be the
largest prime not exceeding $d^{1/(2k)}-1$, and put $H=H_{k,q}$.  Write
$N_q=|V(H)|$ and $\Delta_q=\Delta(H)$, and form $G=B(H)$.  Since $H$ is
$\Delta_q$-regular, $G$ is $(\Delta_q+1)$-regular and
$$|E(G)|=N_q(\Delta_q+1).$$
Moreover,
$$\Delta(G)=\Delta_q+1\le K_q+1
  <(q+1)^{2k}\le d,$$
and Lemma~\ref{lem:edge-reduction} gives
$$\diam L(G)\le k+1=\ell.$$
The prime number theorem and Proposition~\ref{prop:H} give
$$K_q=(1+o(1))d,
  \qquad
  \Delta_q=(1+o(1))d,
  \qquad
  N_q=(1+o(1))\Delta_q^k.$$
Consequently,
$$|E(G)|=N_q(\Delta_q+1)
  =(1+o(1))d^{k+1}=(1+o(1))d^\ell.$$
Since $h_\ell(d)-1\ge|E(G)|$, this proves~\eqref{eq:edge-limit}.

The graph $G$ is bipartite.  Conversely, Theorem~7 (or more precisely, Theorem~10) of~\cite{CCJK} 
implies that every bipartite graph of maximum degree at most $d$ and line-graph diameter at most $\ell$ has at most $d^\ell$ edges.  Thus the maximum number of edges in the bipartite version is $(1+o(1))d^\ell$.
\end{proof}

After this work was completed, we learned that the projective-plane construction in [20, Lemma 3.3] shares the rank-2 incidence graph underlying the case k=1 of our construction. The higher-rank construction and its odd–even routing argument are independent and different.
\paragraph{Tool and computational resource disclosure.}
GPT-5.6 Pro was used during exploratory brainstorming about possible constructions. The discussion was directed by formulating the required parameter scale and proposing higher rank spherical buildings as the geometric setting. The tool suggested splitting complete flags into their odd- and even-rank subflags, initially in connection with the edge problem. The authors developed this suggestion into the halved-flag construction, formulated and verified all mathematical arguments, and take full responsibility for the content. Generative-AI tools assisted in developing the accompanying Lean~4 formalization~\cite{CK26}.


\begin{thebibliography}{99}
\raggedright

\bibitem{AB}
P.~Abramenko and K.~S.~Brown,
\emph{Buildings: Theory and Applications},
Graduate Texts in Mathematics, vol.~248, Springer, 2008.
\href{https://doi.org/10.1007/978-0-387-78835-7}
{doi:10.1007/978-0-387-78835-7}.

\bibitem{ACG}
N.~Alon, F.~R.~K.~Chung, and R.~L.~Graham,
Routing permutations on graphs via matchings,
\emph{SIAM J. Discrete Math.} \textbf{7} (1994), 513--530.
\href{https://doi.org/10.1137/S0895480192236628}
{doi:10.1137/S0895480192236628}.

\bibitem{BSS}
M.~Bachrat\'y, J.~\v{S}iagiov\'a, and J.~\v{S}ir\'a\v{n},
Asymptotically approaching the Moore bound for diameter three by Cayley graphs,
\emph{J. Combin. Theory Ser. B} \textbf{134} (2019), 203--217.
\href{https://doi.org/10.1016/j.jctb.2018.06.003}
{doi:10.1016/j.jctb.2018.06.003}.

\bibitem{BI}
E.~Bannai and T.~Ito,
On finite Moore graphs,
\emph{J. Fac. Sci. Univ. Tokyo Sect. IA Math.} \textbf{20} (1973), 191--208.
\href{https://doi.org/10.15083/00039786}{doi:10.15083/00039786}.

\bibitem{BB81}
J.-C.~Bermond and B.~Bollob\'as,
The diameter of graphs---a survey,
\emph{Congr. Numer.} \textbf{32} (1981), 3--27.
\href{https://inria.hal.science/hal-02426939}
{HAL:hal-02426939}.


\bibitem{Bol}
B.~Bollob\'as,
\emph{Extremal Graph Theory},
London Mathematical Society Monographs, vol.~11, Academic Press, 1978.

\bibitem{Bol_later}
B.~Bollob\'as,
\emph{Random Graphs},
2nd ed., Cambridge Studies in Advanced Mathematics, vol.~73,
Cambridge University Press, Cambridge, 2001.
\href{https://doi.org/10.1017/CBO9780511814068}
{doi:10.1017/CBO9780511814068}.

\bibitem{CCJK}
S.~Cambie, W.~Cames van Batenburg, R.~de Joannis de Verclos, and R.~J.~Kang,
Maximizing line subgraphs of diameter at most $t$,
\emph{SIAM J. Discrete Math.} \textbf{36} (2022), 939--950.
\href{https://doi.org/10.1137/21M1437354}
{doi:10.1137/21M1437354}.

\bibitem{CvBKP20}
W.~Cames van Batenburg, R.~J.~Kang, and F.~Pirot,
Strong cliques and forbidden cycles,
\emph{Indag. Math. (N.S.)} \textbf{31} (2020), 64--82.
\href{https://doi.org/10.1016/j.indag.2019.09.003}
{doi:10.1016/j.indag.2019.09.003}.

\bibitem{CK26}
W.~Cames van Batenburg, S.~Korsky,
\emph{Lean formalization for 
Asymptotically attaining the Moore bound},
GitHub repository, 2026,
\href{https://github.com/woutercvb/wewantmoore}
{\texttt{github.com/woutercvb/wewantmoore}}.

\bibitem{CG}
E.~A.~Canale and J.~G\'omez,
Asymptotically large $(\Delta,D)$-graphs,
\emph{Discrete Appl. Math.} \textbf{152} (2005), 89--108.
\href{https://doi.org/10.1016/j.dam.2005.03.008}
{doi:10.1016/j.dam.2005.03.008}.

\bibitem{CGTT90}
F.~R.~K.~Chung, A.~Gy\'arf\'as, Z.~Tuza, and W.~T.~Trotter,
The maximum number of edges in $2K_2$-free graphs of bounded degree,
\emph{Discrete Math.} \textbf{81} (1990), 129--135.
\href{https://doi.org/10.1016/0012-365X(90)90144-7}
{doi:10.1016/0012-365X(90)90144-7}.

\bibitem{Cranston26}
D.~W.~Cranston,
Coloring, list coloring, and painting squares of graphs
(and other related problems),
\emph{Electron. J. Combin.}, Dynamic Survey DS25, 2nd ed., 2026.
\href{https://doi.org/10.37236/10898}
{doi:10.37236/10898}.

\bibitem{Dam}
R.~M.~Damerell,
On Moore graphs,
\emph{Proc. Cambridge Philos. Soc.} \textbf{74} (1973), 227--236.
\href{https://doi.org/10.1017/S0305004100048015}
{doi:10.1017/S0305004100048015}.

\bibitem{DHJKV26}
E.~Davey, E.~Hurley, R.~de Joannis de Verclos, R.~J.~Kang,
and J.~Volec,
Strong edge-colouring via local flag algebras,
\emph{arXiv preprint}, 2026.
\href{https://arxiv.org/abs/2607.17421}
{arXiv:2607.17421}.

\bibitem{Del1}
C.~Delorme,
Grands graphes de degr\'e et diam\`etre donn\'es,
\emph{European J. Combin.} \textbf{6} (1985), 291--302.
\href{https://doi.org/10.1016/S0195-6698(85)80043-3}
{doi:10.1016/S0195-6698(85)80043-3}.

\bibitem{Del2}
C.~Delorme,
Large bipartite graphs with given degree and diameter,
\emph{J. Graph Theory} \textbf{9} (1985), 325--334.
\href{https://doi.org/10.1002/jgt.3190090304}
{doi:10.1002/jgt.3190090304}.

\bibitem{Erdos88}
P.~Erd\H{o}s,
Problems and results in combinatorial analysis and graph theory,
\emph{Discrete Math.} \textbf{72} (1988), 81--92.
\href{https://doi.org/10.1016/0012-365X(88)90196-3}
{doi:10.1016/0012-365X(88)90196-3}.

\bibitem{HS}
A.~J.~Hoffman and R.~R.~Singleton,
On Moore graphs with diameters 2 and 3,
\emph{IBM J. Res. Develop.} \textbf{4} (1960), 497--504.
\href{https://doi.org/10.1147/rd.45.0497}{doi:10.1147/rd.45.0497}.

\bibitem{KMP26}
H.~Kumar, B.~Mohar, and S.~Pragada,
An improved bound for the strong clique index of graphs,
preprint, 2026.
\href{https://doi.org/10.48550/arXiv.2607.02698}
{doi:10.48550/arXiv.2607.02698}.

\bibitem{MS}
M.~Miller and J.~\v{S}ir\'a\v{n},
Moore graphs and beyond: a survey of the degree/diameter problem,
\emph{Electron. J. Combin.}, Dynamic Survey DS14, 2nd ed., 2013.
\href{https://doi.org/10.37236/35}{doi:10.37236/35}.

\end{thebibliography}
\end{document}